\documentclass[11pt]{article}
\usepackage[margin=1in]{geometry}
\usepackage[numbers,sort&compress]{natbib}
\usepackage{amsmath,amssymb,amsthm,mathtools}
\usepackage{bm}
\usepackage{microtype}
\usepackage[hidelinks]{hyperref}
\usepackage{enumitem}

\newtheorem{theorem}{Theorem}
\newtheorem{proposition}[theorem]{Proposition}

\newtheorem{remark}[theorem]{Remark}
\newtheorem{definition}[theorem]{Definition}

\newcommand{\R}{\mathbb R}

\newcommand{\Z}{\mathbb Z}
\newcommand{\ii}{\mathrm i}
\newcommand{\cL}{\mathcal L}
\newcommand{\cD}{\mathcal D}
\newcommand{\norm}[1]{\lVert #1\rVert}
\newcommand{\Hper}[1]{H^{#1}_{\#}(0,T;H)}
\newcommand{\Lone}{L^1(0,T;H)}
\newcommand{\Om}{\omega}

\author{Giovanni P. Galdi \and Boris Muha \and Justin T. Webster}
\title{From Polynomial Stability to Periodic Well-posedness\\in Partially Dissipative Systems}
\date{}

\begin{document}
\maketitle

\begin{abstract}
\noindent 
The study of resonances (and well-posedness) for complex systems under time-periodic loading is of broad interest in application. The work of Galdi et al.~(2014) connects asymptotic stability of solutions to an unforced Cauchy problem to solvability of the  time-periodic forced problem. Uniform stability of the solution semigroup gives periodic well-posedness for all forces in the natural mild forcing class, whereas strong stability yields only existence of a dense set of forcings for which resonance can be excluded. We address an intermediate regime for polynomial (also: rational or semiuniform) stability. Working with a Fourier decomposition in Hilbert space, we demonstrate that polynomial stability of the semigroup yields an explicit characterization of the dense forcing set on which periodic well-posedness holds. More precisely, resolvent bounds translate directly into certain losses of time derivatives on the forcing required to ensure well-posedness. Our result is motivated by partially dissipative models---including the famous heat-wave interaction problem idealizing fluid-structure interactions, as well as some thermoelastic,  viscoelastic, and weakly damped hyperbolic systems---for which polynomial decay is the natural regime.
\end{abstract}

\section{Introduction}
The emergence of resonance in systems driven by time-periodic forcing is a classical but delicate issue arising in continuum mechanics, control theory, and the analysis of partial differential equation (PDE) systems. Resonance appears within a class of time-periodic well-posedness problems, where the central question is the following: {\em For an admissible class of forcing functions, for which periods $T$ and which forcings $F_T$ does the system admit a unique, continuous $T$-periodic solution?} We may then view resonance as the existence of a pair $(T,F_T)$ for which the response is not periodic and unbounded in time. Thus, for problems which exhibit continuous dependence on the data, periodic well-posedness in a given forcing class is a demonstration of non-resonance therein. This is the point of view adopted in \cite{Galdi}, and it is the one we take here.

The primary motivation for this study comes from {\em partially dissipative systems} (PDS), where dissipation is spatially or systemically localized. One may apply periodic forcing in such a system and ask whether the dissipation is ``strong enough" to preclude resonances associated to the undamped system components. Typical examples include thermoelastic and viscoelastic dynamics, heat-wave (hyperbolic-parabolic) systems, indirectly damped systems, acoustic chambers, and fluid-structure interactions \cite{Galdi}. In these settings, the prominence of damping is intimately connected to the stability type of the unforced dynamics. One rarely expects uniform (or exponential) decay, while the presence of dissipation does typically yield trajectory-by-trajectory (or strong) decay. The natural large-time regime is therefore often that of \emph{polynomial decay}, also called \emph{rational decay} or \emph{semiuniform stability}, which lies between the notions of strong and uniform stability; this note elucidates the corresponding periodic well-posedness picture.

The prevalence of polynomial stability is reflected in a broad literature. Representative works include indirectly damped or locally coupled wave-type systems \cite{ACK,LiHanXu}, wave equations with acoustic or dynamic boundary conditions \cite{WehbeAML}, Bresse and Timoshenko structures \cite{FatoriMonteiro,CuiChai}, thermoelastic systems \cite{Leseduarte,PataSurvey,DellOroLasieckaPata}, heat-wave or heat-structure transmission problems \cite{Zhang2003,RZZ,AvalosTriggiani2013,ALT,MMSW}, and Stokes-Lam\'e fluid-structure interactions \cite{AT2009,AB2015}. Thus, within the broad class of PDS exhibiting polynomial decay, it is natural to ask whether resonance can occur.

The framework of \cite{Galdi} directly connects semigroup stability to periodic well-posedness. Earlier links between dissipation and recurrent response appeared in abstract setting such as \cite{StraskrabaVejvoda1973}, while non-resonance results for dissipative wave equations \cite{Haraux1985} made explicit that sufficient dissipation suppresses resonance. The semigroup-level criterion of \cite{Galdi} (Proposition \ref{prop:Galdi} below) connects uniform or strong stability of the underlying semigroup to specific notions of periodic well-posedness. This yields a compelling criterion for the exclusion of resonance through semigroup stability results. Yet \cite{Galdi} says nothing about the large intermediate regime between strongly and uniformly stable systems, precisely the natural setting for PDS. 

We provide a direct answer in this setting. By the celebrated Proposition \ref{prop:BT} below, polynomial decay of the semigroup is equivalent to polynomial growth of the resolvent on $i\mathbb R$. If one expands the periodic problem in temporal Fourier modes, the resolvent exponent appears clearly as a Sobolev loss of time derivatives in the periodic data-to-solution map. Thus polynomial stability yields a concrete forcing class  admitting periodic well-posedness. As polynomial stability sits between strong and uniform stability, our result below fits naturally between the endpoint cases given in \cite{Galdi}. 

This connection between polynomial stability and Sobolev loss in the periodic setting was suggested by \cite{MMSW}, focusing on a characteristic heat-wave interaction which is known to exhibit polynomial stability \cite{ALT,AvalosTriggiani2013,Zhang2003,RZZ}. In \cite{MMSW}, time-periodic solutions were constructed through control techniques in the time domain, and the admissible forcing classes were identified through observability-type estimates. Derivative losses in this setting are explicitly tied to geometry and coupling, and finite-energy periodic solutions are recovered  from partial dissipation via time differentiations of the system. 

The present note isolates the corresponding abstract mechanism. Our observation is simple but striking: {\em For time-Fourier modes, the derivative loss encoded in polynomial resolvent bounds identifies a forcing class for periodic well-posedness.} We  emphasize that \cite{Galdi}  uses finite trigonometric polynomials as an admissible class for certain structured examples, solving mode-by-mode. Yet in these cases no uniform, quantitative estimate on solutions is provided, hence no explicit ``infinite-mode'' forcing class is obtained. Finally, because we insist on finite-energy solutions, the forcing must move \emph{up} the regularity scale from standard mild forcing in $L^1(0,T;H)$. There are weaker notions of solution beyond the finite-energy class (see \cite[Appendix B]{MMSW}), but understanding those classes---and their relation to resonance---is part of the broader program motivating this note.

\section{Notions of Stability and Periodic Well-posedness}
For the remainder of this note, let $H$ be a Hilbert space and let $A:\cD(A)\subset H\to H$ generate a bounded $C_0$-semigroup $S(t)=e^{At}$ on $H$.
\begin{definition}\label{1}
We say that $S(t)$ is:
\begin{enumerate}[label={\rm(\roman*)},leftmargin=2.2em]
  \item \emph{strongly stable}, if $S(t)x\to 0$ in $H$ for every $x\in H$;
  \item \emph{uniformly stable}, if $\norm{S(t)}_{\cL(H)}\to 0$ as $t\to\infty$;
  \item \emph{polynomially (or rationally) stable of order $1/\alpha$} if $0\in\rho(A)$ and, for some $\alpha>0$,
  \[
  \norm{S(t)A^{-1}}_{\cL(H)}\le C(1+t)^{-1/\alpha},\qquad t\ge 0.
  \]
\end{enumerate}
\end{definition}

Strong stability is common and is often obtained by  ``soft'' methods---compactness, unique continuation, or spectral contradiction arguments. Uniform stability is  stronger and typically requires hard analysis, for instance via the Gearhart-Pr\"uss theorem \cite{Pruss1984} where global control of the resolvent $\rho(A)$ on $i\mathbb R$ must be obtained. Polynomial stability interpolates between these notions and records a quantified derivative loss at the spectral level. In PDE models, this loss is a spectral reflection of a familiar energy-method phenomenon: because dissipation acts only on part of the system, one must differentiate the system or reconstruct the dynamics through the coupling to recover finite-energy level control. The present theorem makes this derivative loss  directly visible at the periodic level.

Periodic well-posedness should be understood relative to a forcing class, $X_T$. Recalling the semigroup generator $A:\cD(A)\subset H\to H$, we  consider the equation:
\begin{equation}\label{equation}
U_t=AU+F.
\end{equation}
Periodic well-posedness amounts to the existence of a unique $T$-periodic mild solution in $C_{\#}([0,T];H)$ for \eqref{equation}.\footnote{We require that the solution must be continuous in time so that the time traces at $t=0,T$ can be interpreted and equated. Mild solutions satisfy the variation of parameters formula \cite{Galdi}.} 
 In this setting there are three natural regimes:
\begin{enumerate}[leftmargin=2.2em]\setlength\itemsep{.01cm}
\item[(a)] unconditional well-posedness; namely, every $F\in X_T$ gives rise to a unique $T$-periodic solution;
\item[(b)] periodic well-posedness for $F$ in an abstract dense subset $\mathcal Q \subset X_T$;
\item[(c)] periodic well-posedness for a dense $\mathcal Q$ which is characterized explicitly (by additional regularity).
\end{enumerate}
Once a problem is periodically well-posed in some class, resonance is excluded therein by the principle of superposition and the boundedness of $S(t)$. This is to say, no other (non-time periodic) solution may grow unboundedly in time away from the unique periodic solution. 

Uniform and strong stability of $S(t)=e^{At}$ in Definition \ref{1} are tied to items (a) and (b) above via Proposition \ref{prop:Galdi} below. As we consider mild solutions in $C_{\#}([0,T];H)$, we will consider the ambient class of forcing functions to be $X_T=L^1(0,T;H)$.\footnote{Note that functions in $X_T$ can be identified with their unique $T$-periodic a.e. extensions to all of $\mathbb{R}$; for instance, $L^1_{\#}(0,T;H)$ denotes the space of $L^1(0,T;H)$ functions extended $T$-periodically to a.e. $t \in \mathbb{R}$.}
\begin{proposition}[Theorems 3.4 and 3.6, \cite{Galdi}] \label{prop:Galdi} 
Assume, additionally, that $S(t)$ is a contraction semigroup. Then:
\begin{enumerate}[label={\rm(\roman*)},leftmargin=2.2em]
\item If $S(t)$ is uniformly stable, then for all $T>0$,  \eqref{equation} is unconditionally periodically well-posed.
\item If $S(t)$ is strongly stable, then for all $T>0$ there exists a dense subset $\mathcal Q\subset \Lone$ such that \eqref{equation} is periodically well-posed.
\end{enumerate}
\end{proposition}

In Hilbert space, the polynomial stability regime is characterized by the celebrated result of Borichev and Tomilov.
\begin{proposition}[Theorem 2.4, \cite{BT}]\label{prop:BT}
Let $\alpha>0$. Then the following are equivalent:
\begin{enumerate}[label={\rm(\alph*)},leftmargin=2.2em]
\item $0\in\rho(A)$ and ~$\norm{S(t)A^{-1}}_{\cL(H)}\le C(1+t)^{-1/\alpha}$ ~for $t\ge 0$;
\item $\ii\R\subset\rho(A)$ and~
$\norm{(\ii sI-A)^{-1}}_{\cL(H)}\le C(1+|s|)^{\alpha},\qquad s\in\R.$
\end{enumerate}
\end{proposition}

The next section provides the missing ``middle case'': a concrete characterization of the forcing set $\mathcal Q \subset \Lone$ 
 on which periodic well-posedness holds under polynomial stability of $S(t)$.

\section{Polynomial Stability Yields A Well-posedness Class}
In the context of $L^2(0,T;H)$-based spaces, the $T$-periodic problem is naturally reduced to the ``sample lattice'' $\{\ii n\omega :n\in\Z\}$, where $\omega=2\pi/T$. In this context, we may work with natural forcing classes of the form $H^m_{\#}(0,T;H)$, namely, the $T$-periodic space of functions whose $m$ distributional time derivatives take values in $L^2(0,T;H)$, $m \ge 0$. The norm on this space can be taken as 
$$
\norm{F}_{H^m_{\#}(0,T;H)}^2 = \sum_{n\in\Z}(1+|n|)^{2m}\norm{F_n}_H^2,$$
with $F_n$ being the $n$th Fourier coefficient at frequency $n\omega$.

Our simple observation is that the spectral polynomial stability requirement gives a quantitative sampled resolvent estimate, and therefore {\em quantitative control of the  Fourier multiplier.} This rather immediately yields the missing ``middle case'' between the two endpoint cases in Proposition \ref{prop:Galdi}. It also makes precise a central, but implicit, point in the direct PDE constructions of \cite{MMSW}:  To recover a finite-energy periodic response, the forcing must be taken in a more regular class than $L^1(0,T;H)$; moreover, that additional regularity is determined by the resolvent growth exponent on $i\mathbb R$.
\begin{theorem}\label{thm:main}
Assume that $S(t)$ is polynomially stable of order $1/\alpha$ for some $\alpha>0$. Then, for every $T>0$ and every $m\ge 1$, \eqref{1} is periodically well-posed for $F\in \Hper{m+\alpha}$. Moreover, we have the data-to-solution estimate
\begin{equation}\label{eq:main-est}
\norm{U}_{\Hper{m}}\le C_T\norm{F}_{\Hper{m+\alpha}}.
\end{equation}
\end{theorem}
\begin{remark} We require that $m\ge 1$ so that $H^m_{\#}(0,T;H) \hookrightarrow C_{\#}([0,T];H)$, and hence time-traces are well-defined and  the periodicity condition $U(T)=U(0)$ is well-defined.\end{remark}
\begin{proof}
By Proposition~\ref{prop:BT}, polynomial stability of $S(t)$ ensures that $\ii\R\subset\rho(A)$ and
\[
\norm{(\ii sI-A)^{-1}}\le C(1+|s|)^{\alpha},\qquad s\in\R.
\]
Hence there is a constant $M_T>0$ such that
\begin{equation}\label{eq:lattice-bound}
\norm{(\ii n\omega I-A)^{-1}}\le M_T(1+|n|)^{\alpha},\qquad n\in\Z.
\end{equation}

Now decompose $F$ into its temporal Fourier modes as
\[
F(t)=\sum_{n\in\Z}F_n e^{\ii n\omega t}\in\Hper{m+\alpha}.
\]
Define the mode-by-mode solvers as
\[
U_n:=(\ii n\omega I-A)^{-1}F_n,\qquad  \text{with} \qquad U(t):=\sum_{n\in\Z}U_n e^{\ii n\omega t}.
\]
Using \eqref{eq:lattice-bound},
\[
(1+|n|)^{2m}\norm{U_n}_H^2\le M_T^2(1+|n|)^{2(m+\alpha)}\norm{F_n}_H^2.
\]
Summing over $n$ yields
\[
\sum_{n\in\Z}(1+|n|)^{2m}\norm{U_n}_H^2
\le M_T^2\sum_{n\in\Z}(1+|n|)^{2(m+\alpha)}\norm{F_n}_H^2,
\]
so $U\in \Hper{m}$, and the convergence of the series respects the $H^m$ topology; thus \eqref{eq:main-est} follows.

To identify $U$ as a mild solution, we introduce the Fourier partial sums
\[
F^N(t):=\sum_{|n|\le N}F_n e^{\ii n\Om t},\qquad U^N(t):=\sum_{|n|\le N}U_n e^{\ii n\Om t}.
\]
Each $U^N$ is a finite-mode periodic solution and satisfies the variation-of-parameters formula exactly on $(0,T)$:
\[
U^N(t)=S(t)U^N(0)+\int_0^t S(t-s)F^N(s)\,ds.
\]

First, it is immediate that $F^N\to F$ in $L^2(0,T;H)$ and $U^N\to U$ in $L^2(0,T;H)$. Since we have $U^N \in H^m_{\#}(0,T;H)$ for $m \ge 1$, the standard Sobolev embedding in Bochner spaces implies that $U^N(s) \to U(s) \in H$ uniformly for $s \in [0,T]$. Then the boundedness of $S(t)$ on $[0,T]$ implies that both terms on the right-hand side converge in $C([0,T];H)$; thus one may pass to the limit in the variation-of-parameters identity. The $T$-periodicity of each $U^N$ is thus preserved in the limit, and hence $U$ is a $T$-periodic mild solution.

For uniqueness, we invoke linearity and consider a mild solution with Fourier coefficients  satisfying
\[
(\ii n\Om I-A)U_n=0,\qquad n\in\Z.
\]
Since $\ii n\Om\in\rho(A)$ for every $n$, one obtains $U_n=0$ for all $n$, hence $U\equiv 0 \in L^2(0,T;H)$.
\end{proof}

We reiterate that Theorem~\ref{thm:main} gives a quantitative improvement of Proposition \ref{prop:Galdi}'s strong-stability conclusion (ii). Namely, polynomially stable semigroups are strongly stable, so what is gained here is an explicit derivative-loss law: The  exponent which governs resolvent growth also governs  Sobolev loss in the periodic data-to-solution map. That loss determines a concrete, dense, explicitly characterized subclass $\mathcal Q := H^{m+\alpha}_{\#}(0,T;H)$ of admissible forcings. The use of $L^2$-based spaces is a consequence of the Fourier method. Since periodic trigonometric polynomials are also dense in $L^1(0,T;H)$ and belong to $H^r_{\#}(0,T;H)$ for every $r\ge 0$, each class $H^{m+\alpha}_{\#}(0,T;H)$ is a concrete dense subclass. 

\section{Application: Canonical Heat-Wave Interactions}
As a concrete illustration of Theorem~\ref{thm:main}, we consider a canonical PDS model given by a coupled heat-wave system. Such models have been studied extensively in several geometric and coupling configurations. Here, heat-wave dynamics serve as a representative example rather than a comprehensive case study. We may isolate a well-known PDS system, invoke an available polynomial stability result, and then apply our Theorem~\ref{thm:main} to obtain a novel periodic well-posedness conclusion for a characterized set of forcings $\mathcal Q$. A more systematic comparison across classes of PDS will be undertaken later.

We focus on the standard interface-coupled configuration in which a heat equation acts on an $n$-dimensional bounded domain with smooth (or sufficiently regular) boundary in dimensions $n=2,3$, a wave equation acts on a comparable $n$-dimensional domain, and the two dynamics are strongly coupled across a common, smooth $(n-1)$-dimensional interface. This is a canonical transmission problem studied, e.g., in \cite{Zhang2003,RZZ,AvalosTriggiani2013,ALT}. Under classical {\em geometric control conditions} on the domains, polynomial decay rates can be obtained for the associated solution semigroup, with the latter yielding a decay law of the form
\[
\|{S(t)A^{-1}}\|_{\mathcal L(H)} \lesssim t^{-1/2};
\]
that is, the exponent $\alpha=2$ in the notation of Definition~\ref{1}. Theorem~\ref{thm:main} therefore gives the periodic implication for \eqref{equation}
\[
F\in H^{m+2}_{\#}(0,T;H)\quad\Longrightarrow\quad U\in H^m_{\#}(0,T;H).
\]
In particular, for finite-energy periodic solutions, one may take $m=1$ and recover periodic well-posedness for forcing in $H^3_{\#}(0,T;H)$, with the interpretation of $U(T)=U(0)$ holding in the sense of $C([0,T];H)$. Noting the Sobolev embedding in one dimension in time variable, one may also be more minimal: $$F\in H^{5/2^+}_{\#}(0,T;H)\quad\Longrightarrow\quad U\in H^{1/2^+}_{\#}(0,T;H).$$ In situations where sharper analysis is available, an optimal rational decay result \cite{ALT} corresponds to $\alpha=1$, yielding the stronger conclusion
$$F\in H^{m+1}_{\#}(0,T;H)\quad\Longrightarrow\quad U\in H^m_{\#}(0,T;H).$$

Several discussion points are in order. First, Theorem~\ref{thm:main} treats the heat-wave forcing monolithically, that is, as a single forcing term with values in the full first-order finite-energy state space. The temporal regularity is therefore imposed globally across all system components. By contrast, in \cite{MMSW}, periodic well-posedness is established by direct time-domain methods and the forcing regularity is measured component-wise. This produces different temporal losses for different components of the coupled system. If one compresses the resulting estimates of \cite{MMSW} into a single global forcing class, the  corresponding well-posedness requirement is  of order  $F \in H^4_{\#}(0,T;H)$, a more stringent class than what is obtained from $m=1$ above. This scenario does not indicate any contradiction, as the authors of \cite{MMSW} explicitly describe several places in their constructions where additional derivatives are lost, for instance through trace/embedding constants, so the numerology in \cite{MMSW} is not expected to be sharp. Secondly, these two approaches to periodic well-posedness operate under rather different geometric assumptions. Going through polynomial stability requires  strong observability estimates, and hence relies on classical (and demanding) geometric control hypotheses. The direct periodic construction in \cite{MMSW}, on the other hand, exploits time-periodicity and treats observability in a different way, which allows  broader and more irregular geometries. Thus the direct approach---rather than passing through polynomial stability---yields a more flexible geometric setting at the price of larger derivative losses. 

We also note that both of the aforementioned approaches concern only temporal loss of derivatives. In many concrete models, temporal regularity may be traded for spatial regularity by using the equations themselves. A sharp result along these lines will be strongly affected by the underlying geometry/interface conditions and by available elliptic regularity results, and will be pursued in later work. 

\section{Conclusions}
We believe the picture in Theorem~\ref{thm:main} is simple but useful. Uniform stability gives unconditional periodic well-posedness; strong stability gives only an abstract dense set of non-resonance forcings; and polynomial stability gives an explicit characterization of that class together with quantified derivative losses. In this way the periodic theory becomes aligned with the large semigroup literature on stability types and decay. More importantly, the theorem clarifies that the loss of derivatives is not a technical artefact of a particular PDE argument: it is encoded in sampled resolvent growth.

From the point of view of resonance, several distinct scenarios may occur in PDS, when components/regions of the system are undamped.
\begin{enumerate}[label={\rm(\roman*)},leftmargin=2.2em]\setlength\itemsep{.005cm}
\item An abstract class for periodic well-posedness may exist, but without any usable description. 
\item Polynomial stability yields a concrete non-resonance class, with a derivative loss dictated by resolvent growth.
\item Observed numerology $(\alpha,m)$ may depend on geometry, interface, and dissipation mechanisms (exactly as in \cite{MMSW}), and these factors dictate the ability to obtain observability estimates. 
\item Outside of finite-loss well-posedness classes, one expects more delicate failures, including: observability breakdown, Fourier divergence, infinite loss of derivatives, and genuine resonance; such failures remain at the heart of the broader PDS resonance problem.
\item  For weaker notions of solution (e.g., very weak solutions \cite{MMSW}), their interpretation from the point of view of resonance is an interesting question for future consideration. 
\end{enumerate}

In light of these points, the present note is one step in a much larger program. Our result here, taken with Proposition \ref{prop:Galdi}, certainly does not settle the  resonance problem; yet, it provides an explicit and broadly applicable bridge that completes the story of stability $\to$ periodic well-posedness. It suggests a direct route by which spectral information about the dynamics can inform the analysis of resonance in PDS.

\vskip.2cm
\noindent {\textbf{Acknowledgments}}: This work was partially supported by the Croatian Science Foundation under project IP-2022-10-2962(B.M.).  Muha was supported by the European Union – NextGenerationEU through the National Recovery and Resilience Plan 2021-2026. Institutional grant of University of Zagreb Faculty of Science IK IA 1.1.3. Impact4Math, and by Croatia-USA bilateral grant “The mathematical framework for the diffuse interface method applied to coupled problems in fluid dynamics”. Galdi was partially supported by NSF DMS-2307811.   Webster was partially supported by NSF DMS-2307538.

\end{document}